\documentclass{article}

\usepackage{amsmath,amssymb,amsthm,bm}
\usepackage[a4paper]{geometry}
\usepackage{graphicx}
\usepackage{algorithm,algorithmic}
\usepackage[bookmarksnumbered,pdfencoding=auto,pdfborder={0 0 0}]{hyperref}
\usepackage{authblk}
\usepackage{xcolor}
\usepackage{booktabs}

\newenvironment{keywords}{\medskip\textbf{Keywords:}}{}
\newenvironment{MSC}{\medskip\textbf{MSC code (2020).}}{}

\graphicspath{{figs/}}

\newtheorem{theorem}{Theorem}
\newtheorem{lemma}{Lemma}

\DeclareMathOperator{\diag}{diag}

\DeclareMathOperator{\Span}{span}

\newcommand{\fro}{\mathsf F}

\newcommand{\mi}{\mathrm{i}}
\newcommand{\mj}{\mathrm{j}}
\newcommand{\mk}{\mathrm{k}}

\newcommand*{\set}[1]{\left\lbrace#1\right\rbrace}
\newcommand*{\mat}[1]{\bm{#1}}
\newcommand*{\vect}[1]{\bm{#1}}
\newcommand*{\herm}{^{\mathsf H}}
\newcommand*{\trans}{^{\top}}
\newcommand*{\win}{\mathsf{win}}

\newcommand{\bmat}[1]{\begin{bmatrix}#1\end{bmatrix}}

\newcommand*{\conj}[1]{\overline{#1}}

\def\adots{\mathinner{\mkern2mu\raise1pt\hbox{.}\mkern2mu
    \raise4pt\hbox{.}\mkern2mu\raise7pt\hbox{.}\mkern1mu}}

\DeclareFontFamily{U}{matha}{}
\DeclareFontShape{U}{matha}{m}{n}{
  <-5.5> matha5 <5.5-6.5> matha6 <6.5-7.5> matha7
  <7.5-8.5> matha8 <8.5-9.5> matha9 <9.5-11> matha10
  <11-> matha12
}{}
\DeclareSymbolFont{matha}{U}{matha}{m}{n}

\DeclareFontFamily{U}{mathx}{}
\DeclareFontShape{U}{mathx}{m}{n}{<-> mathx10}{}
\DeclareSymbolFont{mathx}{U}{mathx}{m}{n}

\DeclareMathDelimiter{\ldbrack}{\mathopen}{matha}{"76}{mathx}{"30}
\DeclareMathDelimiter{\rdbrack}{\mathclose}{matha}{"77}{mathx}{"38}

\newcommand*{\eqclass}[1]{\ldbrack#1\rdbrack}

\definecolor{bkgndcolor}{rgb}{0.78,0.93,0.8}

\begin{document}

\title{On Eigenvector Computation and Eigenvalue Reordering for the
Non-Hermitian Quaternion Eigenvalue Problem}
\author[1]{Zhigang Jia}
\author[2,3]{Meiyue Shao}
\author[4,2]{Yanjun Shao}
\affil[1]{School of Mathematics and Statistics and RIMS, Jiangsu Normal
University, Xuzhou 221116, China}
\affil[2]{School of Data Science, Fudan University, Shanghai 200433, China}
\affil[3]{Shanghai Key Laboratory for Contemporary Applied Mathematics, Fudan
University, Shanghai 200433, China}
\affil[4]{Department of Computer Science, Yale University, New Haven, CT 06511,
USA}

\maketitle

\begin{abstract}
In this paper we present several additions to the quaternion QR algorithm,
including algorithms for eigenvector computation and eigenvalue reordering.
A key outcome of the eigenvalue reordering algorithm is that the aggressive
early deflation (AED) technique, which significantly enhances the convergence
of the QR algorithm, is successfully applied to the quaternion eigenvalue
problem.
We conduct numerical experiments to demonstrate the efficiency and
effectiveness of the proposed algorithms.


\begin{MSC}
65F15, 15A33, 15A18
\end{MSC}
\end{abstract}

\section{Introduction}
\label{sec:introduction}

Let \(\mathbb H=\Span_{\mathbb R}\set{1,\mi,\mj,\mk}\) be the skew field of
quaternions, where the basis satisfies
\[
\mi^2=\mj^2=\mk^2=\mi\mj\mk=-1.
\]
The quaternion (right) eigenvalue problem
\[
\mat A\vect x=\vect x\lambda,
\qquad (\mat A\in\mathbb H^{n\times n},~
\vect x\in\mathbb H^n\setminus\set{0},~
\lambda\in\mathbb H)
\]
arises in a variety of applications, including quantum
mechanics~\cite{Adler1995,JonesSmith2010},
image processing~\cite{LM2004,LS2003}, etc.
In this work, we restrict ourselves to the dense, non-Hermitian case (i.e.,
the matrix \(\mat A\) is a dense, non-Hermitian quaternion matrix).
The non-Hermitian quaternion eigenvalue problem naturally emerges in
non-Hermitian quantum mechanics~\cite{JonesSmith2010}.

The dense non-Hermitian quaternion eigenvalue problem has been studied
in~\cite{BBM1989,JWZC2018}.
In~\cite{BBM1989} the (nonsymmetric) Francis QR algorithm~\cite{Francis1961,
Francis1962} was successfully extended to compute the quaternion Schur
decomposition.
Recently the quaternion QR algorithm was reformulated into a
structure-preserving manner in order to improve performance~\cite{JWZC2018}.
However, a few closely related computational tasks, such as eigenvector
computation and eigenvalue reordering, are not discussed in these works.
Moreover, in the past decades, the Francis QR algorithm has been largely
improved by modern techniques such as multishift QR sweeps and aggressive
early deflation (AED)~\cite{BD1989,BBM2002a,BBM2002b,Byers2007,GKK2010,
GKKS2015,KKS2012}.
These modern techniques have not yet been incorporated into the quaternion QR
algorithm.

In this work we discuss several aspects in the dense, non-Hermitian quaternion
eigenvalue problem that have not been carefully addressed in the existing
literature.
We first establish tools to effectively solve upper triangular quaternion
Sylvester equations.
Then the quaternion Sylvester solvers are adopted to tackle higher level
problems, including the computation of all or selected eigenvectors, the
eigenvalue swapping problem, as well as the application of the AED technique.
Thanks to these developments, we obtain a dense, non-Hermitian quaternion
eigensolver that is more efficient and complete.

This paper is an extension of the undergraduate thesis of the third
author~\cite{Shao2023}.
The rest of the paper is organized as follows.
In Section~\ref{sec:preliminary}, we briefly review some basics of the
quaternion (right) eigenvalue problem.
In Section~\ref{sec:eigenvector}, we discuss quaternion Sylvester equation
solvers and develop algorithms for eigenvector computation for upper
triangular quaternion matrices.
In Section~\ref{sec:swap}, we propose the eigenvalue swapping algorithm and
develop the AED technique.
Numerical experiments are presented in Section~\ref{sec:experiments}.

\section{Quaternion right eigenvalue problem}
\label{sec:preliminary}
In the following we provide a brief review of the quaternion right eigenvalue
problem.
We assume that readers are already familiar with quaternion algebra.

Given a quaternion matrix \(\mat A\in\mathbb H^{n\times n}\), the right
eigenvalue problem is to find a scalar \(\lambda\in\mathbb H\) and a vector
\(\mat x\in \mathbb H^n\setminus\set{0}\) such that
\(\mat A\mat x=\mat x\lambda\).
Recall that any two quaternions \(\xi\) and \(\eta\) are \emph{similar} to
each other if there exists a unit quaternion~\(\omega\) such that
\(\eta=\conj{\omega}\xi\omega\).%
\footnote{A quaternion \(\omega\) is called a \emph{unit quaternion} if it
satisfies \(\lvert\omega\rvert=(\conj\omega\omega)^{1/2}=1\).}
Let \(\eqclass{\xi}\) denote the set of quaternions similar to \(\xi\), and
let~\(\mathbb C_+\) denote the upper half-plane including the real axis.
Then \(\eqclass{\xi}\cap\mathbb C_+\) contains a unique element, denoted
by~\(\xi_c\); see, e.g., \cite[Lemma~2.1]{Zhang1997}.
If \(\xi\) is an eigenvalue of \(\mat A\), then so is any other element of
\(\eqclass{\xi}\).
We call~\(\xi_c\) a \emph{standard eigenvalue} or a \emph{standardized
eigenvalue} of \(\mat A\).
Then every eigenvalue of \(\mat A\) can be standardized.
In the rest of this paper we assume that all eigenvalues are already
standardized unless otherwise specified.

From a numerical perspective, if many or all eigenvalues of \(\mat A\) are of
interest, it is recommended to compute the \emph{Schur decomposition}
\begin{equation}
\label{eq:schur_form}
\mat A=\mat U\mat T\mat U\herm,
\end{equation}
where \(\mat U\in\mathbb H^{n\times n}\) is unitary and
\(\mat T\in\mathbb H^{n\times n}\) is upper triangular with diagonal entries
chosen from~\(\mathbb C_+\)~\cite{Brenner1951}.
The computation of~\eqref{eq:schur_form} can be accomplished by the quaternion
QR algorithm~\cite{BBM1989};
see Algorithm~\ref{alg:quaternion_schur}.

\begin{algorithm}
\caption{Quaternion QR Algorithm}
\label{alg:quaternion_schur}
\begin{algorithmic}[1]
\REQUIRE A quaternion matrix \(\mat A\in\mathbb H^{n\times n}\).
\ENSURE A unitary matrix \(\mat U\) and an upper triangular matrix \(\mat T\)
satisfying~\eqref{eq:schur_form}.

\STATE Reduce \(\mat A\) to an Hessenberg matrix \(\mat H_0\) using unitary
       similarity: \(\mat H_0=\mat U\herm\mat A\mat U\).
\WHILE{not converged}
    \STATE Generate the (real) shifting polynomial \(p_k(\cdot)\) for
           \(\mat H_k\).
    \STATE Update \(\mat H_{k+1}\gets\mat Q_k\herm\mat H_k\mat Q_k\) using an
           implicit QR sweep, where \(\mat Q_k\mat R_k\) is the QR
           factorization of \(p_k(\mat H_k)\).
    \STATE Update \(\mat U\gets\mat U\cdot\mat Q_k\).
\ENDWHILE
\STATE Standardize the diagonal entries of \(\mat H_k\), and set
\(\mat T\gets\mat H_k\).
\end{algorithmic}
\end{algorithm}

We would like to make two remarks here.
First, we shall see in Section~\ref{subsec:motivation} that eigenvector
computation is a non-trivial task due to the non-commutativity of quaternion
algebra.
In practice it is often required to compute a few or all eigenvectors
of~\(\mat A\) after the Schur decomposition~\eqref{eq:schur_form} is
calculated.
However, neither~\cite{BBM1989} nor~\cite{JWZC2018} provided a detailed
discussion of how this can be done.
We shall discuss eigenvector computation in Section~\ref{sec:eigenvector}.
Second, in theory, the diagonal entries of \(\mat T\) can take any prescribed
ordering.
This can be easily shown by induction.
In Section~\ref{sec:swap}, we shall discuss how to reorder the diagonal
entries in the Schur form in a numerically stable manner.

\section{Eigenvector computation from the quaternion Schur form}
\label{sec:eigenvector}

\subsection{Motivation}
\label{subsec:motivation}
For a complex matrix \(\mat A\in\mathbb C^{n\times n}\), if an eigenvalue of
the matrix, \(\lambda\), is known, the corresponding eigenvector \(\mat x\)
can be computed by solving the homogeneous linear equation
\begin{equation}
\label{eq:shifted}
(\mat A-\lambda\mat I)\mat x=\mat 0.
\end{equation}
Since the coefficient matrix \(\mat A-\lambda\mat I\) is singular, we usually
impose an additional normalization assumption on \(\mat x\), e.g.,
\(\mat x(1)=1\), to ensure that the system~\eqref{eq:shifted} has a unique
solution, if \(\lambda\) is of multiplicity one.

The situation becomes more complicated for quaternion matrices.
Let us assume that \(\mat A\in\mathbb H^{n\times n}\) has a known single
eigenvalue \(\lambda\in\mathbb C_+\).
Due to the non-commutativity of quaternion algebra, we need
adjust~\eqref{eq:shifted} to a homogeneous Sylvester equation
\begin{equation}
\label{eq:shifted2}
\mat A\mat x-\mat x\lambda=\mat 0.
\end{equation}
However, there is a new obstacle that in general we \emph{cannot} impose
\(\mat x(i)=1\) for any \(i\), even if we can already ensure \(\mat x(i)\neq0\).
Therefore, equation~\eqref{eq:shifted2} is much more difficult to solve
compared to~\eqref{eq:shifted}.%
\footnote{One way to solve~\eqref{eq:shifted2} is to embed it into a
homogeneous Sylvester equation over \(\mathbb C\).
The price to pay is that the matrix dimension is doubled.}

For instance, consider
\[
\mat A=\bmat{2-\mi-2\mj & -1+\mi+2\mj \\ 2-2\mi-2\mj & -1+2\mi+2\mj}.
\]
The eigenvalues of \(\mat A\) are \(\lambda_1=1\) and \(\lambda_2=\mi\),
with eigenvectors
\[
\mat x_1=\bmat{1 \\ 1}, \qquad \mat x_2=\bmat{1-\mj+\mk \\ 2-\mj+\mk}.
\]
It is impossible to normalize any entry of \(\mat x_2\) to \(1\) or any other
complex number, unless \(\lambda_2\) is allowed to be replaced by a
non-complex one in \(\eqclass{\lambda_2}\).
Without the knowledge of \(\mat x_2\), it is even unclear how to properly
choose an element from \(\eqclass{\lambda_2}\) to make~\eqref{eq:shifted2}
easy to solve.

However, the situation becomes much simpler when \(\mat A\) is in the Schur
form.
In the following, we shall show that the eigenvectors of an upper triangular
quaternion matrix can be easily computed without augmenting the matrix
dimension.

\subsection{Upper triangular Sylvester equations}
We first discuss how to solve upper triangular Sylvester equations that
frequently arise in quaternion eigenvalue problems.
The simplest case is the scalar Sylvester equation.
Lemma~\ref{lem:Sylvester-1x1} characterizes the nondegenerate case.
Although this result is well-known, we provide here a constructive proof that is
suitable for numerical computation.

\begin{lemma}[\cite{Johnson1944}]
\label{lem:Sylvester-1x1}
Let \(\alpha\), \(\beta\), \(\gamma\in\mathbb H\).
Then there exists a unique \(\chi\in\mathbb H\) such that
\(\alpha\chi-\chi\beta=\gamma\) if and only if
\(\eqclass{\alpha}\neq\eqclass{\beta}\).
\end{lemma}
\begin{proof}
Let \(\xi_1\) and \(\xi_2\) be unit quaternions such that
\(\tilde\alpha=\conj\xi_1\alpha\xi_1\in\mathbb C_+\) and
\(\tilde\beta=\conj\xi_2\beta\xi_2\in\mathbb C_+\).
Then the Sylvester equation \(\alpha\chi-\chi\beta=\gamma\) reduces to
\begin{equation}
\label{eq:sylv-scalar-complex}
\tilde\alpha\tilde\chi-\tilde\chi\tilde\beta=\tilde\gamma,
\end{equation}
where \(\tilde\chi=\conj\xi_1\chi\xi_2\) and
\(\tilde\gamma=\conj\xi_1\gamma\xi_2\).
By representing \(\tilde\chi\) and \(\tilde\gamma\), respectively, as
\[
\tilde\chi=\tilde\chi_1+\tilde\chi_2\mj,
\qquad
\tilde\gamma=\tilde\gamma_1+\tilde\gamma_2\mj,
\qquad
(\tilde\chi_1,\,\tilde\chi_2,\,\tilde\gamma_1,\,\tilde\gamma_2\in\mathbb C),
\]
equation~\eqref{eq:sylv-scalar-complex} splits into
\[
(\tilde\alpha-\tilde\beta)\tilde\chi_1=\tilde\gamma_1,
\qquad (\tilde\alpha-\conj{\tilde\beta})\tilde\chi_2=\tilde\gamma_2,
\]
which has a unique solution
\begin{equation}
\label{eq:sylv-scalar-sol}
\tilde\chi_1=\frac{\tilde \gamma_1}{\tilde\alpha-\tilde\beta},
\qquad \tilde\chi_2=\frac{\tilde \gamma_2}{\tilde\alpha-\conj{\tilde\beta}}
\end{equation}
if and only if \(\tilde\alpha\neq\tilde\beta\).
\end{proof}

We shall see later that scalar Sylvester equations are frequently encountered
in dense eigensolvers.
Since we impose that all the eigenvalues are standardized, the case in which
\(\alpha\) and \(\beta\) are complex numbers is of particular interest.
In this case we are already given~\eqref{eq:sylv-scalar-complex}, and can
solve it directly by~\eqref{eq:sylv-scalar-sol} without additional
preprocessing/postprocessing.
The pseudocode of this special case is listed as Algorithm~\ref{alg:Qsylv}.
The algorithm makes full use of the knowledge that \(\alpha\) and~\(\beta\)
are complex numbers, and is much simpler than the algorithm proposed
in~\cite{GIMT2009} which requires solving a \(4\times4\) linear system.

\begin{algorithm}[!tb]
\caption{Scalar Sylvester equation solver.}
\label{alg:Qsylv}
\begin{algorithmic}[1]
\REQUIRE Two complex numbers \(\alpha\), \(\beta\in\mathbb C\) with
\(\alpha\neq\beta\) and \(\alpha\neq\conj{\beta}\), and a quaternion number
\(\gamma=\gamma_1+\gamma_2\mj\in\mathbb H\).
\ENSURE The solution \(\chi\) of \(\alpha\chi-\chi\beta=\gamma\).

\IF{\(\alpha\ne\beta\) \AND \(\alpha\neq\conj{\beta}\)}
  \STATE \(\chi_1\gets\gamma_1/(\alpha-\beta)\),
         \(\chi_2\gets\gamma_2/(\alpha-\conj{\beta})\).
  \STATE \(\chi\gets\chi_1+\chi_2\mj\).
\ELSE
  \STATE Report exception.
\ENDIF
\end{algorithmic}
\end{algorithm}

We then consider the upper triangular Sylvester equation for a vector, i.e.,
\begin{equation}
\label{eq:QTsylv}
\mat T \mat x-\mat x\lambda=\mat b,
\end{equation}
where \(\mat T\in\mathbb H^{n\times n}\) is upper triangular with complex
diagonal entries, and \(\lambda\in\mathbb C\) is not an eigenvalue of
\(\mat T\).
This problem can be easily solved by \emph{back substitution}.
By partitioning \(\mat T \mat x-\mat x\lambda=\mat b\) into
\[
\bmat{\mat T_{1,1} & \mat T_{1,2} \\ \mat 0 & \mat T_{2,2}}
\bmat{\mat x_1 \\ \mat x_2}-\bmat{\mat x_1 \\ \mat x_2}\lambda
=\bmat{\mat b_1 \\ \mat b_2},
\]
where \(\mat T_{2,2}\in\mathbb C\) is a scalar, we obtain
\begin{subequations}
\begin{align}
\mat T_{1,1} \mat x_1-\mat x_1\lambda&=\mat b_1-\mat T_{1,2}\mat x_2,
\label{eq:sylv-triu-1}\\
\mat T_{2,2}\mat x_2-\mat x_2\lambda&=\mat b_2.
\label{eq:sylv-triu-2}
\end{align}
\end{subequations}
Since~\eqref{eq:sylv-triu-2} is a scalar Sylvester equation, we first
use Algorithm~\ref{alg:Qsylv} to compute \(\mat x_2\).
Then~\eqref{eq:sylv-triu-1} becomes an \((n-1)\times(n-1)\) upper triangular
Sylvester equation, which can be solved recursively.
The back substitution algorithm for solving~\eqref{eq:QTsylv} is listed in
Algorithm~\ref{alg:QTsylv}.%
\footnote{Throughout the paper, we use MATLAB's colon notation to represent
submatrices.}

\begin{algorithm}[!tb]
\caption{Back substitution algorithm for upper triangular Sylvester equations.}
\label{alg:QTsylv}
\begin{algorithmic}[1]
\REQUIRE An upper triangular quaternion matrix
\(\mat T\in\mathbb H^{n\times n}\) with standardized eigenvalues, a complex
number \(\lambda\in\mathbb C\) that is not an eigenvalue of \(\mat T\), and a
vector \(\mat b\in\mathbb H^n\).
\ENSURE The solution \(\mat x\in\mathbb H^n\) of
\(\mat T \mat x-\mat x\lambda=\mat b\).
On exit, \(\mat x\) overwrites \(\mat b\).

\FOR{\(i=n\) \TO \(1\)}
  \STATE Solve the scalar Sylvester quaternion equation
         \(\mat T(i,i)\chi-\chi\lambda=\mat b(i)\) by
         Algorithm~\ref{alg:Qsylv}.
  \STATE Set \(\mat b(i)\gets\chi\).
  \STATE Update \(\mat b(1:i-1)\gets\mat b(1:i-1)-\mat T(1:i-1,i-1)\mat b(i)\).
\ENDFOR
\end{algorithmic}
\end{algorithm}

\subsection{Eigenvectors of the quaternion Schur form}
With the help of the upper triangular Sylvester solver, we are now ready to
compute the eigenvectors of the quaternion Schur form
\begin{equation}
\label{eq:schurformT}
\mat T=\bmat{\lambda_1 & t_{1,2} & \cdots & t_{1,n-1} & t_{1,n} \\
0 & \lambda_2 & \cdots & t_{2,n-1} & t_{2,n} \\
\vdots & \vdots & \ddots & \vdots & \vdots \\
0 & 0 & \cdots & \lambda_{n-1} & t_{n-1,n} \\
0 & 0 & \cdots & 0 &\lambda_n}.
\end{equation}
For simplicity, we assume that \(\mat T\) has \(n\) distinct eigenvalues.
Then we can diagonalize \(\mat T\) by computing all eigenvectors \(\mat T\).

A key observation is that the eigenvector corresponding \(\lambda_k\) is of
the form
\[
\mat x_k=[x_1,\dotsc,x_{k-1},1,0,\dotsc,0]\trans,
\]
i.e., the \(k\)th entry of \(\mat x_k\) is \(1\).
To illustrate this, let us partition the \(k\times k\) leading submatrix
of~\(\mat T\) into
\[
\bmat{\mat T_{1,1} & \mat T_{1,2} \\ \mat 0 & \lambda_k}.
\]
Let \(\mat y\) be the unique solution of the Sylvester equation
\begin{equation}
\label{eq:sylv-lambdai}
\mat T_{1,1}\mat y-\mat y\lambda_k=-\mat T_{1,2}.
\end{equation}
Setting \([x_1,\dotsc,x_{k-1}]\gets\mat y\trans\) yields
\(\mat T\mat x_k=\mat x_k\lambda_k\).
As the proof is constructive, we formulate it as
Algorithm~\ref{alg:eigenvectorsofT}.
We remark that in practice appropriate scaling is needed to avoid unnecessary
overflow;
see, e.g., \texttt{CTREVC}\slash\texttt{ZTREVC} in LAPACK~\cite{LAPACK}.

\begin{algorithm}[!tb]
\caption{Eigenvector computation of the Schur form.}
\label{alg:eigenvectorsofT}
\begin{algorithmic}[1]
\REQUIRE An upper triangular matrix \(\mat T\in\mathbb H^{n\times n}\)
with \(n\) distinct standardized eigenvalues.
\ENSURE A nonsingular upper triangular matrix
\(\mat X\in\mathbb H^{n\times n}\) such that \(\mat X^{-1}\mat T\mat X\) is
diagonal.

\STATE Set \(\mat X(:,1)\gets\mat e_1.\)
\FOR{\(k=2\) \TO \(n\)}
  \STATE Set \(\mat T_{1,1}\gets\mat T(1:k-1,1:k-1)\),
         \(\mat T_{1,2}\gets\mat T(1:k-1,k)\), and
         \(\lambda_k\gets\mat T(k,k)\).
  \STATE Solve the Sylvester equation~\eqref{eq:sylv-lambdai} by
         Algorithm~\ref{alg:QTsylv}.
  \STATE Set \(\mat X(1:k-1,k)\gets\mat y\), \(\mat X(k,k)\gets1\),
         \(\mat X(k+1:n,k)\gets\mat0\).
\ENDFOR
\end{algorithmic}
\end{algorithm}

\section{Eigenvalue swapping and aggressive early deflation}
\label{sec:swap}
In this section, we present an eigenvalue swapping algorithm to reorder the
diagonal entries in the Schur form.
As an important application of the eigenvalue swapping algorithm, we show that
the AED technique carries over to the quaternion QR algorithm.

\subsection{Eigenvalue swapping algorithm}
\label{subsec:swap}
In principle, we can prescribe any ordering of eigenvalues in the Schur form.
However, like the usual QR algorithm for real or complex matrices, the
quaternion QR algorithm does not have much control over the sequence of
deflated eigenvalues.
Hence, in practice, we often need to reorder the eigenvalues after the Schur
form is calculated.
This is achieved by repeatedly swapping consecutive diagonal entries in the
Schur form.
Theorem~\ref{thm:swap} ensures that eigenvalue swapping can be easily
accomplished.

\begin{theorem}
\label{thm:swap}
Let
\[
\mat T=\bmat{t_{1,1} & t_{1,2} \\ 0 & t_{2,2}}\in\mathbb H^{2\times 2},
\]
where \(t_{1,1}\), \(t_{2,2}\in\mathbb C_+\).
Then there exists a unitary matrix \(\mat Q\in\mathbb H^{2\times2}\) such that
\[
\mat Q\herm\mat T\mat Q=\bmat{t_{2,2} & t_{1,2} \\ 0 & t_{1,1}}.
\]
\end{theorem}

\begin{proof}
The case that \(t_{1,1}=t_{2,2}\) is trivial because we can simply choose
\(\mat Q=\mat I_2\).
In the following we assume that \(t_{1,1}\neq t_{2,2}\).

According to Lemma~\ref{lem:Sylvester-1x1}, there exists a unique solution
\(\chi\in\mathbb H\) of the Sylvester equation
\[
t_{1,1}\chi-\chi t_{2,2}=-t_{1,2}.
\]
This Sylvester equation can be equivalently reformulated as
\[
\bmat{t_{1,1} & t_{1,2} \\ 0 & t_{2,2}}\bmat{\chi \\ 1}
=\bmat{\chi \\ 1}t_{2,2},
\quad\text{or}\quad
[1,-\chi]\bmat{t_{1,1} & t_{1,2} \\ 0 & t_{2,2}}=t_{1,1}[1,-\chi].
\]
Define the unitary matrix
\begin{equation}
\label{eq:Givens-G}
\mat G=\bmat{c & -s \\ s & \conj c},
\end{equation}
where
\begin{equation}
\label{eq:Givens-cs}
s=\bigl(1+\lvert\chi\rvert^2\bigr)^{-1/2}, \qquad c=s\chi.
\end{equation}
Then we have
\begin{equation}
\label{eq:swap}
\mat G\herm\mat T\mat G=\bmat{t_{2,2} & \tilde t_{1,2} \\ 0 & t_{1,1}},
\end{equation}
where \(\tilde t_{1,2}=t_{2,2}\conj\chi-\conj\chi t_{1,1}\).

In order to further transform \(\tilde t_{1,2}\) to \(t_{1,2}\),
we express \(\chi\) as \(\chi=\chi_1+\chi_2\mj\) such that \(\chi_1\),
\(\chi_2\in\mathbb C\), and choose
\(\mat Q=\mat G\cdot\diag\set{\mu_1,\mu_2}\), where
\begin{align*}
\mu_1&=\mi\cdot\exp\bigl(-\mi\cdot\arg(\chi_1)-\mi\cdot\arg(t_{1,1}-\conj t_{2,2})\bigr),\\
\mu_2&=\mi\cdot\exp\bigl(\mi\cdot\arg(\chi_1)-\mi\cdot\arg(t_{1,1}-\conj t_{2,2})\bigr).
\end{align*}
Here, the notation \(\arg(\cdot)\) represents the argument of a complex number, and \(\arg(0)\) is set to \(0\).
It can then be verified that \(\mat Q\herm \mat Q=\mat I_2\) and
\[
\mat Q\herm \mat T \mat Q=\bmat{t_{2,2} & t_{1,2} \\ 0 & t_{1,1}}.
\qedhere
\]
\end{proof}

From a computational perspective, it makes more sense to use \(\mat G\)
instead of \(\mat Q\) for eigenvalue swapping.
Strictly preserving the entry \(t_{1,2}\) is often not worth the cost to
calculate \(\mu_1\) and \(\mu_2\).
Therefore, we propose Algorithm~\ref{alg:swap} based on~\eqref{eq:swap}.

A straightforward application of Algorithm~\ref{alg:swap} is to move a few
selected eigenvalues to the top-left corner of the Schur form \(\mat T\) and
form an orthonormal basis of the corresponding invariant `subspace' (more
rigorously, the invariant right \(\mathbb H\)-submodule).
A more advanced application is the AED technique, which will be discussed in
the subsequent subsection.

\begin{algorithm}[!tb]
\caption{Eigenvalue swapping algorithm}
\label{alg:swap}
\begin{algorithmic}[1]
\REQUIRE An upper triangular matrix \(\mat T\in\mathbb H^{2\times2}\) with
standardized eigenvalues.
\ENSURE A unitary matrix \(\mat G\) that swaps \(\mat T(1,1)\) and
\(\mat T(2,2)\).
On exit, \(\mat T\) is overwritten by \(\mat G\herm\mat T\mat G\).

\IF{\(\mat T(1,1)=\mat T(2,2)\)}
  \STATE Set \(\mat{G}\gets\mat{I}_2\).
\ELSE
  \STATE Solve the scalar Sylvester equation
         \(\mat T(1,1)\chi-\chi \mat T(2,2)=-\mat T(1,2)\) by
         Algorithm~\ref{alg:Qsylv}.
  \STATE Set \(\mat G\) according to~\eqref{eq:Givens-G}
         and~\eqref{eq:Givens-cs}.
  \STATE Swap \(\mat T(1,1)\leftrightarrow\mat T(2,2)\) and set
         \(\mat T(1,2)\gets t_{2,2}\conj\chi-\conj\chi t_{1,1}\).
\ENDIF
\end{algorithmic}
\end{algorithm}

\subsection{Quaternion aggressive early deflation}
\label{subsec:AED}
Aggressive early deflation (AED) is a modern technique proposed by Braman,
Byers, and Mathias to significantly enhance the convergence of the QR
algorithm~\cite{BBM2002b,Byers2007,GKK2010,GKKS2015,KKS2012,Kressner2008}.
Given an unreduced upper Hessenberg matrix \(\mat H\), the AED technique
consists of the following three stages;
see also Figure~\ref{fig:AED}.
\begin{enumerate}
\item
\textbf{Schur decomposition}:
Compute the Schur decomposition of the trailing \(n_{\win}\times n_{\win}\)
submatrix of \(\mat H\) (we call this submatrix the \emph{AED window}), and
apply the corresponding unitary transformation to \(\mat H\).
This produces a `spike' of dimension \(n_{\win}\) to the left of the AED
window.

\item
\textbf{Convergence test}:
If the bottom entry of the `spike' has a tiny magnitude, we replace it by zero
and deflate the corresponding eigenvalue (i.e., the diagonal entry of
\(\mat H\));
otherwise, the eigenvalue is marked as undeflatable and is moved towards the
top-left corner of the AED window by repeatedly applying the eigenvalue
swapping algorithm (i.e., Algorithm~\ref{alg:swap}).
Repeat this convergence test until all eigenvalues of the AED window are
either deflated or marked as undeflatable.

\item
\textbf{Hessenberg reduction}:
Reduce the undeflatable part of \(\mat H\) back to the upper Hessenberg form.
\end{enumerate}

\begin{figure}[!tb]
\centering
\includegraphics[width=\textwidth]{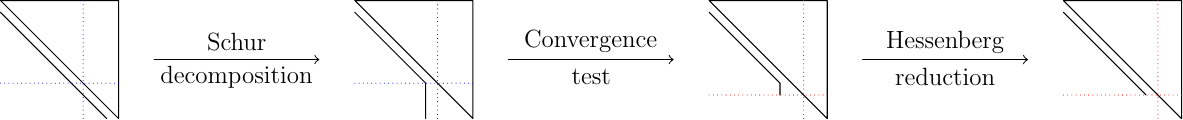}
\caption{A visual illustration of AED.}
\label{fig:AED}
\end{figure}

Because algorithms for Hessenberg reduction and Schur decomposition already
exist for quaternion matrices, we have gathered all the necessary building
blocks for performing AED for quaternion matrices with the help of the
eigenvalue swapping algorithm.

We note that in a multishift variant of the quaternion QR algorithm, the
undeflatable eigenvalues left by AED can be used as the shifts.
However, the development of the small-bulge multishift QR algorithm is beyond
the scope of this work.

Although the AED technique was developed to enhance the convergence of the
small-bulge multishift QR algorithm~\cite{BBM2002a}, according
to~\cite{Kressner2005}, even the very traditional Francis QR algorithm can be
accelerated by AED.
It is natural to ask if the AED technique remains effective when applied to
quaternion matrices.
Fortunately, existing theoretical analyses on AED~\cite{BBM2002b,Kressner2008}
suffice to predict the effectiveness of AED in the quaternion QR algorithm.
In fact, by embedding the quaternion matrix into a complex matrix with double
the dimensions, the AED technique---essentially the extraction of Ritz
pairs---has more opportunities to detect converged eigenvalues compared to
classical deflation based on testing subdiagonal entries.

Finally, we remark that the eigenvalue swapping algorithm also allows us to
develop other advanced algorithms.
For example, the Krylov--Schur algorithm for solving large-scale eigenvalue
problems~\cite{Stewart2002}, which is closely related to
AED~\cite{Kressner2005,Kressner2008}, also carries over to quaternion
matrices.

\section{Numerical experiments}
\label{sec:experiments}
In this section, we conduct some numerical experiments to evaluate the
effectiveness of the proposed algorithms.
These experiments were carried out using MATLAB 2023a and the \texttt{QTFM
toolbox} version~3.4~\cite{QTFM}, on a Linux cluster comprising ten 48-core
\texttt{Intel Xeon Gold 6342} CPUs, each with a frequency of 2.80~GHz, and
20~GB of memory.
We ran our MATLAB programs on one CPU in this cluster.
All computations were performed using IEEE double-precision floating-point
numbers, with a machine precision of
\(\epsilon=2^{-52}\approx2.22\times10^{-16}\).

We define two classes of randomly generated quaternion matrices as follows.
\begin{itemize}
\item
\texttt{fullrand}: A dense square matrix whose entries are randomly generated.
\item
\texttt{hessrand}: A dense upper Hessenberg matrix whose nonzero entries are
randomly generated.
\end{itemize}
These class of matrices are frequently used to test non-Hermitian dense
eigensolvers~\cite{GKK2010,GKKS2015,KKS2012}.
For each nonzero entry, we generate a random unit quaternion \(\omega\) using
the \texttt{randq()} function from the \texttt{QTFM toolbox}, and then
multiply it by a random real number \(\alpha\) uniformly distributed in the
range of \([0,1]\).
The resulting value of \(\omega\cdot\alpha\) is then used as the matrix entry.

\subsection{Performance tests}
In the following we test the effectiveness of the AED technique in the
quaternion QR algorithm~(i.e., Algorithm~\ref{alg:quaternion_schur}).
The size of the AED window, \(n_{\win}\), is adaptively adjusted based on the
matrix size \(n\), following the strategy used in LAPACK's
\texttt{IPARMQ}~\cite{Byers2007}.
The threshold for skipping a QR sweep, commonly known as `\texttt{NIBBLE}', is
set to \(14\%\), the default value used in LAPACK's
\texttt{IPARMQ}~\cite{Byers2007}.

Tables~\ref{tab:fullrand} and~\ref{tab:hessrand} present the detailed results
for \texttt{fullrand} and \texttt{hessrand} matrices, respectively, with
matrix dimensions ranging from \(64\times64\) to \(1024\times1024\).
The tables contain the total number of QR sweeps, total execution time (in
seconds), time spent constructing the Schur vectors \(\mat Q\), and time spent
on AED for each test case.

The results show that the quaternion QR algorithm incorporated with AED
consistently performs fewer QR sweeps and has a shorter execution time than
the original quaternion QR algorithm.
The reduction in QR sweeps becomes more significant as the matrix size
increases (see Figure~\ref{fig:iter}), leading to a substantial decrease in
total execution time.

Although our implementations are only preliminary ones in MATLAB, we expect
that the AED technique remains highly effective in a high performance
implementation of the quaternion QR algorithm, since the number of QR sweeps
can be significantly reduced.

\begin{table}[!tb]
\centering
\small
\caption{Performance tests on the QR algorithm, with and without AED, for
\texttt{fullrand} matrices.}
\label{tab:fullrand}
\begin{tabular}{cccccc}
\toprule
strategy & matrix size & total QR sweeps & total time & time to construct Q & time for AED \\
\midrule
QR+AED & 64   & 173  & $6.88\times10^0$ & $7.08\times10^{-1}$ & $3.34\times10^{0}$ \\
QR     & 64   & 200  & $4.11\times10^0$ & $8.12\times10^{-1}$ & N/A \\
QR+AED & 128  & 267  & $2.39\times10^1$ & $2.73\times10^{0}$  & $1.13\times10^{1}$ \\
QR     & 128  & 399  & $1.95\times10^1$ & $4.26\times10^{0}$  & N/A \\
QR+AED & 256  & 420  & $9.66\times10^1$ & $1.26\times10^{1}$  & $4.05\times10^{1}$ \\
QR     & 256  & 784  & $1.25\times10^2$ & $2.81\times10^{1}$  & N/A \\
QR+AED & 512  & 647  & $6.86\times10^2$ & $1.20\times10^{2}$  & $2.21\times10^{2}$ \\
QR     & 512  & 1530 & $1.61\times10^3$ & $3.82\times10^{2}$  & N/A \\
QR+AED & 1024 & 935  & $8.63\times10^3$ & $1.75\times10^{3}$  & $1.60\times10^{3}$ \\
QR     & 1024 & 3095 & $2.13\times10^4$ & $5.23\times10^{3}$  & N/A \\
\bottomrule
\end{tabular}
\end{table}

\begin{table}[!tb]
\centering
\small
\caption{Performance tests on the QR algorithm, with and without AED, for
\texttt{hessrand} matrices.}
\label{tab:hessrand}
\begin{tabular}{cccccc}
\toprule
strategy & matrix size & total QR sweeps & total time & time to construct Q & time for AED \\
\midrule
QR+AED & 64   & 159  & $7.84\times10^0$ & $7.41\times10^{-1}$ & $4.14\times10^0$ \\
QR     & 64   & 202  & $4.13\times10^0$ & $8.15\times10^{-1}$ & N/A \\
QR+AED & 128  & 262  & $2.86\times10^1$ & $3.02\times10^{0}$  & $1.49\times10^1$ \\
QR     & 128  & 406  & $2.06\times10^1$ & $4.43\times10^{0}$  & N/A \\
QR+AED & 256  & 330  & $8.49\times10^1$ & $8.82\times10^{0}$  & $4.57\times10^1$ \\
QR     & 256  & 880  & $1.41\times10^2$ & $3.11\times10^{1}$  & N/A \\
QR+AED & 512  & 427  & $6.12\times10^2$ & $9.33\times10^{1}$  & $2.59\times10^2$ \\
QR     & 512  & 1925 & $2.20\times10^3$ & $5.13\times10^{2}$  & N/A \\
QR+AED & 1024 & 919  & $9.52\times10^3$ & $1.87\times10^{3}$  & $2.02\times10^3$ \\
QR     & 1024 & 3915 & $2.57\times10^4$ & $6.32\times10^{3}$  & N/A \\
\bottomrule
\end{tabular}
\end{table}

\begin{figure}
\centering
\includegraphics[width=\textwidth]{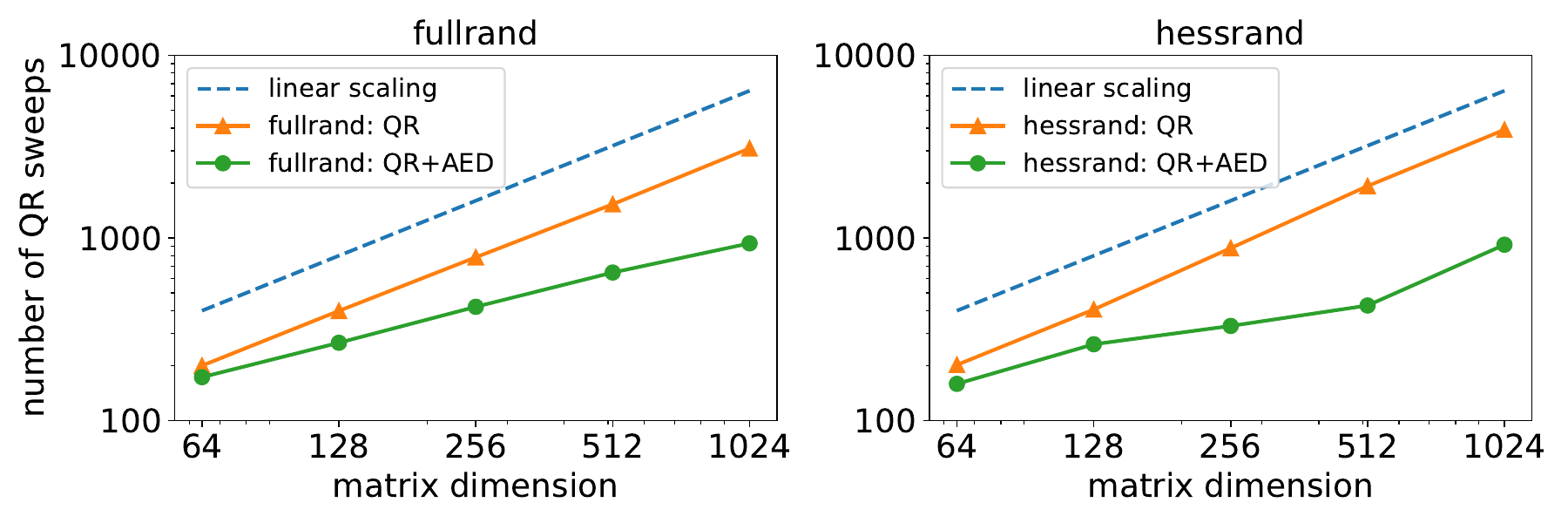}
\caption{The number of QR sweeps with respect to the matrix dimension.}
\label{fig:iter}
\end{figure}

\subsection{Stability tests}
In the following, we evaluate the backward stability of the QR algorithm, as
well as the eigenvector computation.
For the Schur decomposition \(\mat A=\mat Q\mat T\mat Q\herm\) and the
spectral decomposition \(\mat A=\mat X\mat\Lambda\mat X^{-1}\), we measure the
following quantities:
\[
e_1=\frac{1}{\sqrt n}\lVert\mat Q\herm\mat Q-\mat I\rVert_{\fro},
\qquad
e_2=\frac{\lVert\mat Q\herm\mat A\mat Q-\mat T\rVert_{\fro}}
{\lVert\mat A\rVert_{\fro}},
\qquad
e_3=\frac{\lVert\mat A\mat X-\mat X\mat\Lambda\rVert_{\fro}}
{(\lVert\mat A\rVert_{\fro}+\lVert\mat\Lambda\rVert_{\fro})
\lVert\mat X\rVert_{\fro}}.
\]
The results for \texttt{fullrand} and \texttt{hessrand} matrices are listed in
Tables~\ref{tab:fullrand_err} and~\ref{tab:hessrand_err}, respectively.
We can see that both the AED strategy and the eigenvector computation are
numerically stable.
In most test cases, when the AED strategy is incorporated, the backward errors
are slightly lower.
This is likely because AED effectively reduces the total number of
floating-point operations.

\begin{table}[!tb]
\centering
\caption{Stability tests on the QR algorithm, with and without AED, for
\texttt{fullrand} matrices.}
\label{tab:fullrand_err}
\begin{tabular}{ccccc}
\toprule
strategy & matrix size & \(e_1\) & \(e_2\) & \(e_3\) \\
\midrule
QR+AED & 64   & $9.2\times10^{-15}$ & $6.4\times10^{-15}$ & $6.4\times10^{-16}$ \\
QR     & 64   & $9.0\times10^{-15}$ & $6.4\times10^{-15}$ & $7.2\times10^{-16}$ \\
QR+AED & 128  & $1.3\times10^{-14}$ & $8.5\times10^{-15}$ & $6.9\times10^{-16}$ \\
QR     & 128  & $1.3\times10^{-14}$ & $9.2\times10^{-15}$ & $7.0\times10^{-16}$ \\
QR+AED & 256  & $1.7\times10^{-14}$ & $1.1\times10^{-14}$ & $6.0\times10^{-16}$ \\
QR     & 256  & $1.7\times10^{-14}$ & $1.2\times10^{-14}$ & $6.6\times10^{-16}$ \\
QR+AED & 512  & $2.1\times10^{-14}$ & $1.3\times10^{-14}$ & $5.1\times10^{-16}$ \\
QR     & 512  & $2.5\times10^{-14}$ & $1.7\times10^{-14}$ & $6.9\times10^{-16}$ \\
QR+AED & 1024 & $2.5\times10^{-14}$ & $1.6\times10^{-14}$ & $4.3\times10^{-16}$ \\
QR     & 1024 & $3.4\times10^{-14}$ & $2.5\times10^{-14}$ & $6.8\times10^{-16}$ \\
\bottomrule
\end{tabular}
\end{table}

\begin{table}[!tb]
\centering
\caption{Stability tests on the QR algorithm, with and without AED, for
\texttt{hessrand} matrices.}
\label{tab:hessrand_err}
\begin{tabular}{ccccc}
\toprule
strategy & matrix size & \(e_1\) & \(e_2\) & \(e_3\) \\
\midrule
QR+AED & 64   & $1.0\times10^{-14}$ & $6.1\times10^{-15}$ & $3.9\times10^{-16}$ \\
QR     & 64   & $8.8\times10^{-15}$ & $6.0\times10^{-15}$ & $4.4\times10^{-16}$ \\
QR+AED & 128  & $1.3\times10^{-14}$ & $8.0\times10^{-15}$ & $2.9\times10^{-16}$ \\
QR     & 128  & $1.3\times10^{-14}$ & $8.7\times10^{-15}$ & $2.8\times10^{-16}$ \\
QR+AED & 256  & $1.7\times10^{-14}$ & $1.0\times10^{-14}$ & $1.7\times10^{-16}$ \\
QR     & 256  & $1.8\times10^{-14}$ & $1.3\times10^{-14}$ & $2.1\times10^{-16}$ \\
QR+AED & 512  & $2.2\times10^{-14}$ & $1.2\times10^{-14}$ & $1.2\times10^{-16}$ \\
QR     & 512  & $2.7\times10^{-14}$ & $1.8\times10^{-14}$ & $8.8\times10^{-17}$ \\
QR+AED & 1024 & $2.3\times10^{-14}$ & $9.2\times10^{-15}$ & $4.8\times10^{-17}$ \\
QR     & 1024 & $3.6\times10^{-14}$ & $2.3\times10^{-14}$ & $5.8\times10^{-17}$ \\
\bottomrule
\end{tabular}
\end{table}

\section{Conclusions}
\label{sec:conclusions}

In this paper, we discuss several aspects of the dense non-Hermitian
quaternion eigenvalue problem.
We develop algorithms for eigenvector computation from the Schur form, and
the eigenvalue swapping algorithm.
As an application of the eigenvalue swapping algorithm, we discuss the
aggressive early deflation (AED) technique for the quaternion QR algorithm.
These developments fill the gap in existing dense quaternion eigensolvers---we
have come to a point where no theoretical obstacle remains for the non-Hermitian
quaternion QR algorithm.
What is left in this direction is mainly the work of high performance
computing---how to implement efficient dense quaternion eigensolvers.
This includes, but not limited to, the development of efficient quaternion
BLAS libraries~\cite{GV2008,WL2019}, efficient Hessenberg
reduction~\cite{KK2011,KK2012,QV2006}, small-bulge multishift QR
algorithm~\cite{BBM2002a,Byers2007,GKK2010,GKKS2015}, level-3 eigenvalue
reordering algorithm~\cite{GKK2009,Kressner2006}, and level-3 eigenvector
computation~\cite{SKK2020}.

\section*{Acknowledgment}
Z. Jia is supported by the National Key R\&D program of China
(No.~2023YFA1010101), the National Natural Science Foundation of
China (Nos.~12090011, 12171210), and the Major Projects of Universities in
Jiangsu Province (No.~21KJA110001).


\end{document}